\documentclass{amsart}
\usepackage{amsfonts}
\usepackage{amscd}
\usepackage{amssymb}

\usepackage{amscd}
\usepackage{amssymb}

\newtheorem{theorem}{Theorem}[section]
\newtheorem{lemma}[theorem]{Lemma}

\theoremstyle{definition}

\theoremstyle{remark}
\newtheorem{remark}[theorem]{Remark}

\newtheorem{proposition}[theorem]{Proposition}

\newcommand{\R}{{\mathbb R}}
\newcommand{\Z}{{\mathbb Z}}
\newcommand{\C}{{\mathbb C}}
\newcommand{\N}{{\mathbb N}}
\newcommand{\X}{{\mathbb X}}
\newcommand{\Y}{{\mathbb Y}}

\newcommand{\Ca}{{\mathbb O}}
\newcommand{\Ha}{{\mathbb H}}
\newcommand{\Fi}{{\mathbb F}}

\newcommand{\Sch}{{\mathcal C^2}}
\newcommand{\A}{{\mathcal A}}

\renewcommand{\S}{{\mathcal S}}

\renewcommand{\Im}{{\mathop{\mathrm{Im}}\,}}

\newcommand{\arccosh}{{\mathop{\mathrm{arccosh}}\,}}

\def\SO{\mathrm{SO}}

\def\Sp{\mathrm{Sp}}
\def\U{\mathrm{U}}

\def\Spin{\mathrm{Spin}}

\def\sh{\sinh}
\def\ch{\cosh}

\def\es{e^s}
\def\e-s{e^{-s}}

\def\de-s{e^{-2s}}

\numberwithin{equation}{section}

\begin{document}

\title[Cuspidal discrete series]
{Cuspidal discrete series\\for projective hyperbolic spaces}
\author{Nils Byrial Andersen}
\address{Department of Mathematics,
Aarhus University,
Ny Munkegade 118,
Building 1530,
DK-8000 Aarhus C,
Denmark}
\email{byrial@imf.au.dk}
\author{Mogens Flensted--Jensen}
\address{Department of Mathematical Sciences,
University of Copenhagen, Universitetspar\-ken 5,
DK-2100 Copenhagen {\O}, Denmark}
\email{mfj@life.ku.dk}



\dedicatory{Dedicated to Sigurdur Helgason on the occasion of his $85$th birthday}

\begin{abstract}
We have in \cite{AF-JS} proposed a definition of cusp forms on semisimple symmetric spaces $G/H$,
involving the notion of a Radon transform and a related Abel transform. For the real non-Riemannian
hyperbolic spaces, we showed that there exists an infinite number of
cuspidal discrete series, and at most finitely many non-cuspidal discrete series, including in particular the spherical discrete series.
For the projective spaces, the spherical discrete series are the only non-cuspidal discrete series.
Below, we extend these results to the other hyperbolic spaces, and
we also study the question of when the Abel transform of a Schwartz function is again a Schwartz function.
\end{abstract}

\maketitle

\section{Introduction}

We initiated, in joint work with Henrik Schlichtkrull, in \cite{AF-JS}
a generalization of Harish-Chandra's notion of cusp forms for real
semisimple Lie groups $G$ to semisimple symmetric spaces $G/H$.
In the group case, all the discrete series are cuspidal,
and this plays an important role in Harish-Chandra's work on the Plancherel formula.
However, in the established generalizations to $G/H$, cuspidality plays no role and, in fact, was hitherto not defined at all.

The notion of cuspidality relates to the integral geometry on the
symmetric space by using integration over a certain unipotent
subgroup $N^*\subset G$, whose definition is given in \cite{AF-JS}.
The map $f\mapsto \int_{N^*} f(\cdot nH)\,dn$,
which maps functions on $G/H$ to functions on $G/N^*$, is a kind of a {\it Radon transform} for $G/H$.
A discrete series is said to be {\it cuspidal} if it is annihilated
by this transform.

Let $p,q$ denote positive integers. The Radon transform, and the question of cuspidality,
on the real hyperbolic spaces $\SO(p,q+1)_e/\SO(p,q)_e$, was treated in detail in \cite{AF-JS}.
We showed that there is at most a finite number of non-cuspidal discrete series, including in particular
all the spherical discrete series, but also some non-spherical
discrete series. The non-spherical non-cuspidal discrete series are given 
by odd functions on the real hyperbolic space,
which means that they do not descend to functions on the real projective hyperbolic space.

In the present paper, we consider the projective hyperbolic spaces over the classical fields
$\Fi = \R, \C, \Ha$,
$$
G/H = \mathrm{O}(p+1,q+1) / (\mathrm{O}(p+1,q)\times \mathrm{O}(1)), \, \U(p+1,q+1) / (\U(p+1,q) \times \U(1)),
$$
$$
\Sp(p+1,q+1) / (\Sp(p+1,q) \times \Sp(1)),
$$
for $p\ge 0,q\ge 1$. Notice the change of indices from $p$ to $p+1$, to simplify formulae and calculations.

Our main result, Theorem~\ref{Main-thm}, states that the non-cuspidal discrete series for the projective
hyperbolic spaces precisely consist of the spherical discrete series.
The Radon transform of the generating functions is also given explicitly.
Finally, we show that the Abel transform maps (a dense subspace of) the Schwartz functions on $G/H$ perpendicular
to the non-cuspidal discrete series into Schwartz functions.
The latter result also holds for the non-projective real case, and is a new result for all cases.

Our calculations and main results are also valid, with $p=0, \, q=1$ and $d=8$, for the Cayley numbers $\Ca$,
corresponding to the exceptional symmetric space $F_{4(-20)} / \Spin(1,8)$.
Although the model for this space, and the group action on it, is more complicated,
this space can for our purposes be viewed as
$$
F_{4,(-20)} / \Spin(1,8) \, = "\U(1,2; \Ca) / \U(1,1; \Ca)\times \U(1; \Ca)".
$$
We state our results in full generality, but only give complete proofs for the
non-exceptional projective spaces, with some remarks on the other cases in the last section.

We would like to thank Henrik Schlichtkrull for input and fruitful discussions,
which in the real case lead to the explicit formulae involving the Hypergeometric function.
We also want to thank Job Kuit for discussions of part (vi) of Theorem~\ref{Main-thm}, explaining how to prove a
similar result in split rank one, using general theory.

Part of this work was outlined by the first author at the 
Special Session `Radon Transforms and Geometric Analysis in Honor of Sigurdur Helgason's 85th Birthday', 
at the 2012 AMS National Meeting in Boston.
He is grateful to the organizers Jens Christensen, Fulton Gonzalez, and Todd Quinto, for their
invitation to speak, and the hospitality at the meeting, and the subsequent Workshop on Geometric Analysis on Euclidean and Homogeneous Spaces.

\section{Model and structure}

Let $\mathbb F$ be one of the classical fields $\R$, $\C$ or
$\mathbb H$, and let $x\mapsto \overline{x}$ be the standard
(anti-) involution of $\mathbb F$. We make the standard identifications between
$\C$ and $\R^2$, and between $\mathbb H$ and $\R^4$.
Let $p\ge 0,\,q \ge 1$ be two integers, and consider the Hermitian form
$[ \cdot ,\cdot ]$ on $\mathbb F ^{p+q+2}$ given by
\[
[ x,y ]  =x_1 \overline{y}_1 + \dots + x_{p+1} \overline{y}_{p+1} - x_{p+2} \overline{y}_{p+2} -
\dots - x_{p+1+q+1} \overline{y}_{p+1+q+1},\,\, (x,y \in \mathbb F ^{p+q+2}).
\]
Let $G = \U (p+1,q+1;\mathbb F)$ denote the group of $(p+q+2)\times
(p+q+2)$ matrices over $\mathbb F$ preserving $[ \cdot ,\cdot ]$.
Thus $\U(p+1,q+1;\R)= \mathrm{O}(p+1,q+1)$,
$\U(p+1,q+1;\C)= \U(p+1,q+1)$ and $\U(p+1,q+1;\mathbb H)= \Sp(p+1,q+1)$
in standard notation.
Put $\U (p;\mathbb F) = \U (p,0;\mathbb F)$.

Let $x_0 = (0,\dots,0,1)^T$, where superscript $T$ indicates transpose.
Let $H=\U (p+1,q;\mathbb F)\times \U(1;\mathbb F)$ be the subgroup of $G$ stabilizing the line $\mathbb F\cdot x_0$
in $\mathbb F^{p+q+2}$. An involution $\sigma$ of $G$ fixing $H$ is given by
$\sigma (g) =JgJ$, where $J$ is the diagonal matrix with entries $(1,\dots,1,-1)$.
The reductive symmetric space $G/H$ (of rank $1$) can be identified with
the projective hyperbolic space $\X = \X (p+1,q+1;\mathbb F)$:
\[
\X= \{ z\in \mathbb F^{p+q+2} : [z,z] = -1\}/\sim,
\]
where
$\sim$ is the equivalence relation $z\sim zu,\,u\in \mathbb F^*$.

The Lie algebra $\mathfrak g$ of $G$ consists of $(p+q+2)\times
(p+q+2)$ matrices
\[\mathfrak g =
\left(
\begin{array}
[c]{cc}%
A & B \\
B^* & C \\
\end{array}
\right),
\]
where $A$ is a skew Hermitian $(p+1)\times (p+1)$ matrix,
$C$ is a skew Hermitian $(q+1)\times (q+1)$ matrix, and $B$ is an arbitrary
$(p+1)\times (q+1)$ matrix. Here $B^*$ denotes the conjugated transpose of $B$.

Let $K= K_1 \times K_2 =\U(p+1;\mathbb F)\times \U(q+1;\mathbb F)$
be the maximal compact subgroup of $G$ consisting of elements fixed by the
classical Cartan involution on $G$, $\theta (g)=(g^*) ^{-1},\,g\in G$.
Here $g^*$ denotes the conjugated transpose of $g$.
The Cartan involution on $\mathfrak g$ is given by: $\theta (X)=-X^*$. Let
$\mathfrak g = \mathfrak k \oplus \mathfrak p$ be the decomposition of
$\mathfrak g$ into the $\pm 1$-eigenspaces of $\theta$, where
$\mathfrak k = \{X\in \mathfrak g : \theta (X) = X\}$ and $\mathfrak p = \{X\in \mathfrak g : \theta (X) = -X\}$.
Similarly, let $\mathfrak g = \mathfrak h \oplus \mathfrak q$ be the decomposition of
$\mathfrak g$ into the $\pm 1$-eigenspaces of $\sigma(X) = JXJ$, where
$\mathfrak h = \{X\in \mathfrak g : \sigma (X) = X\}$ and $\mathfrak q = \{X\in \mathfrak g : \sigma (X) = -X\}$.

We choose a maximal abelian subalgebra
$\mathfrak a_{\mathfrak q}\subset \mathfrak p\cap \mathfrak q$ as
\[\mathfrak a_{\mathfrak q} = \left \{X_{t_1}=
\left(
\begin{array}
[c]{ccc}%
0 & 0 & t_1\\
0 & 0_{p,q} & 0 \\
t_1 & 0 & 0
\end{array}
\right) : t_1\in \R\right\},
\]
where $0_{p,q}$ is the $(p+q)\times (p+q)$ null matrix.  The exponential of $X_{t_1}$, $a_{t_1} = \exp (X_{t_1})$, is given by
\[a_{t_1} =\exp (X_{t_1})=
\left(
\begin{array}
[c]{ccc}%
\cosh t_1 & 0 & \sinh t_1\\
0 & I_{p,q} & 0 \\
\sinh t_1 & 0 & \cosh t_1
\end{array}
\right) ,
\]
where $I_{p,q}$ is the $(p+q)\times (p+q)$ identity matrix.
Also define $A_{\mathfrak q}=\exp (\mathfrak a_{\mathfrak q})$.

Let $A_{\mathfrak q}^+ = \{a_{t_1} : t_1 >0\}$.
Let $a(x) = a(kah)=a$ denote the projection onto
the $\overline{A_{\mathfrak q}^+}$ component in the Cartan decomposition $G=K\overline{A_q^+}H$ of $G$.
Let $M$ be the centralizer of $X_1 \in \mathfrak a_{\mathfrak q}$ (i.e.,\,when $t_1 =1$) in $K\cap H$.
Then $M$ is the stabilizer of the line $\mathbb F (1,0,\dots, 0,1)$, and the homogeneous space
$K/M$ can be identified with the projective image $\Y = \Y _{p+1,q+1}$ of the product of unit spheres
$\mathbb S^{p} \times \mathbb S^{q}$:
\[
\Y =\{ y\in \mathbb F^{p+q+2} : |y_1|^2 + \cdots + |y_{p+1}|^2=
|y_{p+2}|^2+ \cdots + |y_{p+q+2}|^2=1\}/\sim.
\]
The image of the set
$\{ z\in \mathbb F^{p+q+2} : [z,z] = -1,\,(z_{1},\dots,z_{p+1}) \ne 0 \}$ in $\X$ is
an open dense subset, which we will denote by $\X '$.
The map
\[
K/M \times \R^+ \to \X,\, (kM,t_1) \mapsto ka_{t_1} H,
\]
is a diffeomorphism onto $\X'$.

We introduce spherical coordinates on $\X$ as the pull back of the map:
$$
x(t_1,y) = (u \sinh t_1 ; v \cosh t_1 ),\,t_1 \in \R _+ ,\, y = (u;v) \in \mathbb S^{p} \times \mathbb S^{q}.
$$
We define a ($K$-invariant) `distance' from $x\in \X$ to the origin
as $|x| = |x(t_1,y)| =|t_1| $. Then $\X ' = \{x\in \X|\,|x| >0\}$.
We note that $\cosh ^2( |x|) = |x_{p+2}|^2 +\cdots + |x_{p+q+2}|^2$.
For $g\in G$, we define $|g| = |gH|$.

Let $r= \min\{p,q\}$, and let $X_t$ be the $(r+1)\times(r+1)$ anti-diagonal matrix with entries $t= (t_1,\dots,t_{r+1})\in \R^{r+1}$,
starting from the upper right corner.
We extend $\mathfrak a_{\mathfrak q}$ (viz.\,as $t_2=\cdots =t_{r+1}=0$) to a maximal subalgebra $\mathfrak a\subset \mathfrak p$ as
\[
\mathfrak a = \left \{ X_{t}=
\left(
\begin{array}
[c]{ccc}%
0 & 0 & X_t\\
0 & 0 & 0 \\
X_t^*  & 0 & 0
\end{array}
\right)
\right \}.
\]
We will also consider the sub-algebra $\mathfrak a_{\mathfrak h}= \mathfrak a\cap{\mathfrak h}
= \{ X_t \in \mathfrak a : t_1 = 0\}$.

Let (considered as row vectors)
$$u=(u_1,\dots,u_{p})\in \mathbb F^{p}\quad\text{and}\quad
v = (v_q, \dots ,v_1)\in\mathbb F^q.$$
It turns out to be convenient to number the entries of $v$ from right to left
as indicated.
Let furthermore $w\in \Im \mathbb F$ (i.e.,\,$w = 0$ for $\mathbb F = \R$).
Now define $N_{u,v,w}\in \mathfrak g$ as the matrix given by
\[
N_{u,v,w}=\left(
\begin{array}
[c]{cccc}%
-w & u & v & w\\
-\overline u^T & 0  & 0 & \overline u^T\\
\overline{v}^T & 0 & 0 & -\overline{v}^T \\
-w & u & v & w
\end{array}
\right).
\]
Then $\exp (N_{u,v,w}) =I+ N_{u,v,w}+ 1/2 N_{u,v,w} ^2$, and
\begin{equation}\label{expN}
\exp (N_{u,v,w}) \cdot x_0
= (1/2 (|u|^2 - |v|^2)+w,\overline u^T; -\overline{v}^T , 1 +1/2 (|u|^2 - |v|^2) +w)^T.
\end{equation}
A small calculation also yields that
\begin{equation}\label{aexpN}
a_{t_1}\exp ( N_{u,v,w}) \cdot x_0=\qquad\qquad\qquad\qquad\qquad\qquad\qquad\qquad
\end{equation}
\begin{equation*}
(\sinh t_1+ 1/2 e^{t_1}(|u|^2 - |v|^2)+e^{t_1}w, \overline u^T; -\overline{v}^T, \cosh t_1 + 1/2 e^{t_1} (|u|^2 - |v|^2) +e^{t_1}w)^T,
\end{equation*}
for any $t_1 \in \R$.

We note that $[X_{t_1},N_{u,v,0}] = t_1 N_{u,v,0}$, and $[X_{t_1},N_{0,0,w}] = 2t_1 N_{0,0,w}$.
Let $\gamma (X_{t_1}) = t_1$.
Then the root system $\Sigma _{\mathfrak q}$ for $\mathfrak a_{\mathfrak q}$ 
is given by $\Sigma _{\mathfrak q} = \{ \pm \gamma\}$, for
$\mathbb F = \R$, and $\Sigma _{\mathfrak q} = \{ \pm \gamma\}\cup \{ \pm 2\gamma\}$, for
$\mathbb F = \C,\,\mathbb H$.
The associated nilpotent subalgebra ${\mathfrak n}_{\mathfrak q}$ is given by
${\mathfrak n}_{\mathfrak q} = \mathfrak g ^\gamma
 = \{N_{u,v,0} : u \in \mathbb F^p,\,v\in \mathbb F ^q\}$, when $\mathbb F = \R$,
and ${\mathfrak n} _{\mathfrak q}= \mathfrak g ^\gamma + \mathfrak g ^{2\gamma}
 = \{N_{u,v,w} : u \in \mathbb F^p,\,v\in \mathbb F ^q,\,w\in \Im \mathbb F\}$,
when $\mathbb F = \C,\,\mathbb H$.
Half the sum of the positive roots, $\rho _{\mathfrak q}= 
\frac 12 \sum _{\alpha \in \Sigma _{\mathfrak q} ^+} m_\alpha \alpha$, where
$m_\alpha$ is the multiplicity of the root $\alpha$, is thus
\[
\langle \rho _{\mathfrak q}, X_{t_1}\rangle = \frac 12(dp+dq + 2(d-1))t_1,
\]
where $d= \dim _{\R} \mathbb F$. Using the identification $A_{\mathfrak q} \sim \R$,
we will also sometimes use the definition $\rho _{\mathfrak q} = \frac 12(dp+dq + 2(d-1))\in \R$.

The (restricted) $\Sigma$ for $\mathfrak a$ is given by $\{\pm t_i \pm t_j\},\,
i\ne j,\,i,j \in \{1,\dots, r+1\}$, $\{\pm t_i\} ,\,i \in \{1,\dots, r+1\}$, if $p\ne q$, and
$\{\pm 2 t_i\} ,\,i \in \{1,\dots, r+1\}$, if $d\ge 2$.
Let $\alpha _{i,j}(X_t) = t_i +t_j,\,i< j$, $\beta _{i,j}(X_t) = t_i -t_j,\,i< j$, and
$\gamma _i (X_t) = t_i$.

We choose two sets of positive roots
\begin{equation*}
\Sigma ^+ = \{\alpha _{i,j},\,\beta _{i,j},\,\gamma _i,\,2\gamma _i\},
\end{equation*}
which corresponds to the (standard) ordering $t_1> t_2> \cdots > t_{r+1}$, and
\begin{equation*}
\Sigma _1^{+} = \{\alpha _{i,j},\,\gamma _i,\,2\gamma _i\}
\cup \{\beta _{i,j} : i \ne 1\}
\cup \{-\beta _{i,j} : i = 1\},
\end{equation*}
which corresponds to the ordering $ t_2 >t_3 >\cdots > t_{r+1}>t_1$. The double roots $\{\pm 2 \gamma_i\}$
are not present for $\mathbb F = \R$, the single roots $\{\pm  \gamma_i\}$
are not present when $p=q$. The associated
nilpotent subalgebras are denoted by ${\mathfrak n}$ and ${{\mathfrak n}} _1$ respectively.
The half sum of positive roots $\rho _1$ with regards to $\Sigma _1^{+}$ is given by (restricted to $A_{\mathfrak q}$)
\begin{equation*}
\langle \rho _1, X_{t_1}\rangle  = \frac 12((|dp-dq|+2(d-1))t_1.
\end{equation*}
As before, we will sometimes use the definition $\rho _{1} = \frac 12((|dp-dq|+2(d-1))\in \R$.

We note that $\gamma \in\Sigma _{\mathfrak q}^+$ is the restriction of the roots
$\{\alpha _{1,1+j},\,\gamma _1,\,\beta _{1,1+j}\}$, with $j\in\{1,\dots, r\}$,
where
\[
\mathfrak g ^{\alpha _{1,j+1}}
 = \{N_{u,v,0} : u _j = -\overline v_j,\,u_i= v_i = 0,\,i\ne j\},
\]
\[
\mathfrak g ^{\beta _{1,j+1}}
 = \{N_{u,v,0} : u _j = \overline v_j,\,u_i= v_i = 0,\,i\ne j\},
\]
and, for $p>q$,
\[
\mathfrak g ^{\gamma _{1}}
 = \{N_{u,v,0} : u = (0,\dots, 0,u_{q+1},\dots, u_p),\,v=0\},
\]
which for $p<q$ becomes $u=0,\,v= (v_q, v_{q-1},\dots,v_{p+1}, 0,\dots, 0)$.

Define $\mathfrak n^* ={\mathfrak n} _1 \cap \mathfrak n_{\mathfrak q}$ as the subalgebra
associated to the roots
$\{\alpha _{1,1+j},\,\gamma _1,\,2\gamma _1\}$. Then, for $p\ge q$
\begin{equation} \label{nforpgeq}
\mathfrak n^* = \{N_{u,v,w} : u =
(-\overline{v^r},u'),\,v\in \mathbb F ^{q},\,u'\in \mathbb F^{p-q} \},
\end{equation}
and, for $p<q$,
\begin{equation}\label{nforplq}
\mathfrak n^* = \{N_{u,v,w} : v =
(-\overline{u^r},v'),\,u\in \mathbb F ^{p},\,v' \in \mathbb F^{q-p} \},
\end{equation}
where $u^r,v^r$ means that the order of the indices is reversed. In the following
we shall by abuse of notation leave out the ${}^r$.

\begin{remark}
We have the identity $\Sigma _{\mathfrak q} ^+
 = \{\alpha \in \Sigma ^+:
\alpha _{|{\mathfrak a}_{\mathfrak q}} >0\}_{|{\mathfrak a}_{\mathfrak q}}$.
We then have the disjoint union $\Sigma _{\mathfrak q} ^+ =
\Sigma ^{++} \cup \Sigma  ^{+0}\cup \Sigma  ^{+-}$, where the second sign refers to
$\alpha _{|{\mathfrak a}_{\mathfrak h}}$.
The choice of the nilpotent subalgebra $\mathfrak n^*$
can thus be described by the correspondence $\mathfrak n^* \sim\Sigma ^{++} + \Sigma  ^{+0}$.
\end{remark}

\section{The discrete series}

From \cite[Section 8]{FJ} and \cite[Table 2]{FJO}, we have the following parametrization of the discrete series for the projective
hyperbolic spaces, with an exception for $q=d=1$:
\begin{equation*}
\{T_\lambda \,|\, \lambda = \frac 12 (dq-dp)-1+\mu _\lambda >0, \mu_\lambda \in 2\Z\}.
\end{equation*}
The spherical discrete series are given by the parameters $\lambda$ for which $\mu_\lambda\le 0$, including
the 'exceptional' discrete series corresponding to the (finitely many) parameters $\lambda>0$ for which $\mu_\lambda<0$.
We notice that spherical discrete series exists if, and only if, $d(q-p)>2$.
For $q=d=1$, the discrete series is parameterized by $\lambda \in \R\backslash\{0\}$ such that
$|\lambda|+\rho_{\mathfrak q}\in 2\Z$, and there are no spherical discrete series.

The parameter $\lambda$ is, via the formula
$\Delta f=(\lambda^2-\rho_{\mathfrak q} ^2)f$, related to the eigenvalue of the
Laplace-Beltrami operator $\Delta$ of $G/H$ on  functions $f$ in the
corresponding representation space in $L^2(G/H)$
(with suitable normalization of~$\Delta$). Using \cite[Theorem~5.1]{FJO} (see \cite[Proposition~3.2]{AF-JS} for more details),
we can explicitly describe the discrete series
by generating functions $\psi _\lambda$ as follows.
Let $s =s_1\in  \R$ describe the elements $a_s=a_{s_1}\in A_{\mathfrak q}$. Let $\lambda$ be a discrete series parameter.
For $\mu_\lambda \ge 0$, we have
\begin{equation*}
\psi _\lambda (k a_s H) =\psi _\lambda (x(s,y))=
\phi_{\mu _\lambda}(k) (\cosh s)^{-\lambda -\rho_{\mathfrak q}},
\end{equation*}
where $\phi_{\mu _\lambda}$ is a $K\cap H$-invariant zonal spherical function, in particular $\phi_{0}=1$.
For $\mu_\lambda=-2m \le 0$, we have
$$ \psi_\lambda (k a_s H) =
P_\lambda(\cosh ^2s) (\cosh s)^{-\lambda-\rho_{\mathfrak q}-2m},
$$
where $P_\lambda$ is a polynomial of degree $m$. For $q=d=1$, 
consider the one-parameter subgroup $T=\{k_\theta\}\subset K_2$ defined by
\[k_{\theta} =
\left(
\begin{array}
[c]{cccc}%
I_{p+1} & 0 &  0 \\
0& \cos\theta  & \sin\theta\\
0& -\sin\theta & \cos\theta
\end{array}
\right) ,
\]
where $I_j$ denotes the identity matrix of size $j$,
then $\psi _\lambda (k_\theta a_s H)=
e^{im \theta} (\cosh s)^{-|\lambda| -\rho_{\mathfrak q}}$, with
$m=\lambda \pm \rho_{\mathfrak q}$, and the sign determined by the sign of $\lambda$. See \cite[Section~3]{AF-JS} for further details.

\section{Schwartz functions}

In this section we recall some results from \cite[Chapter\,17]{vdb} regarding
$L^2$-Schwartz functions on $G/H$. Let $\Xi$ denote Harish-Chandra's bi-$K$-invariant
elementary spherical function $\varphi_0$ on $G$, and define the real analytic function
$\Theta: G/H\to\R ^+$ by
\begin{equation*}
\Theta (x) =
\sqrt{\Xi(x\sigma(x)^{-1})}\qquad (x\in G).
\end{equation*}
We notice that there exists a positive constant $C$, and a positive integer $m$, such that
\begin{equation}\label{HCest}
a^{-\rho_{\mathfrak q}} \le \Theta (a) \le C a^{-\rho _{\mathfrak q}} ( 1 +|a|)^m,\qquad
(a\in \overline{A_{\mathfrak q}^+}).
\end{equation}
Here we use the definition $a^\lambda =  e^{\langle \lambda ,\log a\rangle}$, for $a\in {A_{\mathfrak q}^+},\,
\lambda \in {\mathfrak a_{\mathfrak q}}_\C^*$.

The space $\Sch (G/H)$ of $L^2$-Schwartz functions on $G/H$ can be defined as the space of all smooth
functions on $G/H$ satisfying
\begin{equation*}
\mu ^2 _{n,D} (f)= \sup _{x\in G/H} \Theta ^{-1} (x) ( 1 +|x|)^n
|f(D,x)| < \infty,
\end{equation*}
for all $n\in \N\cup \{0\}$ and $D\in U(\mathfrak g)$.

Let $f\in \Sch(G/H)$.
Let $S\subset G$ be a compact set. Then, for any $n\in \N\cup \{0\}$, there exists a positive
constant $C$, such that
\begin{equation} \label{Schestimate}
|f(g\cdot x)| \le C\, \Theta (a(x))   ( 1 +|x|)^{-n},\qquad (g\in S,\,x\in G/H).
\end{equation}

\section{A Radon transform and an Abel transform}

Let $N^* = \exp(\mathfrak n^*)$ and $N_1 =\exp({\mathfrak n} _1)$ denote
the two nilpotent subgroups generated by $\mathfrak n^*$ and
${\mathfrak n} _1$ respectively.
For functions on $G/H$ we define, assuming convergence,
\begin{equation} \label{defradon} Rf(g) = \int _{N^*} f(g n^*H)\, dn^*  \qquad (g\in G).
\end{equation}
Let $H^A$ denote the centralizer of $A$ in $H$. Then $Rf(gm)=Rf(g), \, m\in H^A$, and

\begin{theorem}\label{thm1}
Let $f\in \Sch(G/H)$.
\begin{enumerate}
\item[(i)] The integral defining the Radon transform $R$ converges uniformly on compact sets.

\item[(ii)] $Rf \in C^\infty (G/H^A N_1)$.

\item[(iii)] The Radon transform is $G$- and $\mathfrak g$-equivariant.
\end{enumerate}
\end{theorem}
\begin{proof}
We first assume $p\ge q$.
Let $f\in \Sch(G/H)$, and fix a compact set $S\subset G$. Let $n\in \N$. Then
\begin{equation}\label{1}
\int _{N^*} |f(gn^*H)| dn^*\le C \int _{N^*} a(n^*)^{-\rho_{\mathfrak q}} (1+|n^*|)^{-n+m} dn^*
\qquad (g\in S),
\end{equation}
for the constants $C$ and $m$ given by (\ref{HCest}) and (\ref{Schestimate}).
From (\ref{expN}) and (\ref{nforpgeq}), we have
\[
\cosh ^{2} (| \exp (N _{(v,u'),v,w})|)
= (1+ 1/2 |u'| ^2)^2 +|v|^2 + |w|^2.
\]
Using that $\log s \le\arccosh  s\le \log s + \log 2$, when $s\ge 1$, we see that
the last integral in (\ref{1}) is bounded by
\begin{align*}
 C\int _{\R ^{dp-dq} \times \R^{dq}\times \R^{d-1}}& ((1+1/2 |u'| ^2)^2 +|v|^2+|w|^2)^{-\frac {dp+dq+2(d-1)}4}\\
 &\times
(1+ \log ((1+1/2 |u'| ^2)^2 +|v|^2+|w|^2))^{-n+m} du' dv dw,
\end{align*}
where $C$ is a positive constant.

Consider the integral ($x\in\R ^{k},\,y\in \R^{l}$), with $n>2$,
\begin{equation*}
 \int _{\R ^{k} \times \R^{l}} (1+ |x| ^4+ |y|^2)^{-a}
(1+ \log (1+ |x| ^4+ |y|^2))^{-n} dx dy.
\end{equation*}
With the substitution $y = \sqrt {1+|x| ^4} z \in \R ^l$,  we get
\begin{align*}
& \int _{\R ^{k}\times \R^{l}} (1+ |x| ^4)^{-a+\frac l2}(1+ |z| ^2)^{-a}
(1+ \log (1+ |x| ^4)+ \log (1+|z|^2))^{-n} dx dz\le \\
&\int _{\R ^{k}}(1+ |x| ^4)^{-a+\frac l2}(1+ \log (1+ |x| ^4))^{-\frac n2} dx
\int _{\R ^{l}} (1+ |z| ^2)^{-a}(1+ \log (1+ |z| ^2))^{-\frac n2}dz,
\end{align*}
which is finite if, and only if, $k\le 4a-2l$ and $l\le 2a$.

We have $k=dp-dq,\,l=dq+d-1$ and $a = (dp+dq+2(d-1))/4$,
whence $k = 4a-2l$ and $l\le 2a$, and the integral
(\ref{defradon}) converges uniformly on compact sets.

In the $p<q$ case, we see from (\ref{expN}) and (\ref{nforplq}), that
$\cosh ^{2} (|\exp (N _{u,(u,v'),w})|)=
 |v'|^2 + |u^2| +(1-1/2 |v'|^2)^2 + |w|^2 = 1+|u|^2+1/4 |v'|^4 +|w|^2$,
and we proceed as before, reversing the roles of $u$ and $v$.
\end{proof}

We define the Abel transform $\A$ by $\A f(a)=a^{\rho_1} Rf(a)$, for $a\in A_{\mathfrak q}$.

\begin{theorem}\label{thm2}
Let $g\in G$ and $f\in \Sch(G/H)$.
Let $\Delta$ denote the Laplace--Beltrami operator on $G/H$ and let
$\Delta _{A_{\mathfrak q}}$ denote the
Euclidean Laplacian on $A_{\mathfrak q}$.
Then
\begin{equation}\label{exchange}
\A (\Delta f) = (\Delta _{A_{\mathfrak q}}- \rho_{\mathfrak q} ^2)\A f \qquad (a\in A_{\mathfrak q}).
\end{equation}
\end{theorem}
\begin{proof}
See \cite[Lemma 2.4]{AF-JS}, and the discussion before and after this lemma.
\end{proof}

Let $\psi _\lambda$ belong to the discrete series with parameter $\lambda$. Since
$\Delta\psi_\lambda= (\lambda^2-\rho_{\mathfrak q}^2)\psi_\lambda $,
we see that $\A \psi_\lambda$ is an eigenfunction for the Euclidean Laplacian
$\Delta _{A_{\mathfrak q}}$ on $A_{\mathfrak q}$ with the eigenvalue $\lambda ^2$.
This implies in particular that $s\mapsto R\psi_\lambda(a_s)$ is a linear combination of
$e^{(\lambda-\rho _1)s}$ and $e^{(-\lambda-\rho _1)s}$.

\section{The main result}

Here we state the main theorem, to be proven in the following sections.
We will in particularly be interested in the values of
$Rf$ on the elements $a_s\in A_{\mathfrak q}$, so for simplicity we write
$Rf(s) = Rf(a_s)$, and, similarly, $\A f(s)= \A f (a_s)$.

Let $R>0$, and let $C_R^\infty (G/H)$ denote the subspace of smooth functions on $G/H$
with support inside the ($K$-invariant) `ball' of radius $R$.
Let similarly $C_R^\infty (\R)$ denote the subspace of smooth functions on $\R$
with support inside $[-R,R]$. Finally, let $ \S (\R)$ denote the Schwartz functions on $\R$.

\begin{theorem}\label{Main-thm}
Let $G/H$ be a projective hyperbolic space over $\R, \, \C, \, \Ha$, with $p\ge 0, \, q\ge 1$, or over $\Ca$, with  $p=0,\, q=1$.
\begin{enumerate}
\item[(i)] If $d(q-p) \le 2$, then all discrete series are cuspidal.
\item[(ii)]If $d(q-p) > 2,$ then non-cuspidal discrete series exists,
given by the parameters $\lambda > 0$ with $\mu_\lambda \le 0$.
More precisely, if $0 \ne f\in \Sch(G/H)$ belongs to $T_\lambda$, then $\A f(s)=C e^{\lambda s}$, with $C \ne 0$.
\item[(iii)] $T_\lambda$ is non-cuspidal if and only if $T_\lambda$ is spherical.
\item[(iv)] If $p\ge q$, and $f\in C^\infty _R (G/H)$, for $R>0$, then $\A f \in C_R^\infty (\R)$.
\item[(v)] If $d(q-p) \le 1$, and $f\in \Sch(G/H)$, then $\A f \in \S (\R)$.
\item[(vi)] Assume $d(q-p) > 1$. Let $D$ be the $G$-invariant differential operator
$\Delta_\rho (\Delta_\rho - \lambda_1^2) \dots (\Delta_\rho - \lambda_r^2)$,
where $\lambda_1, \dots, \lambda_r$ are the parameters of the non-cuspidal discrete series,
and $\Delta_\rho = \Delta + \rho_{\mathfrak q}^2$. Then $\A(Df) \in \S (\R)$, for $f$ in a dense subspace of 
$\Sch(G/H)$.
\end{enumerate}
\end{theorem}

\begin{remark}
The theorem also holds for the non-projective spaces $SO(p+1,q+1)_e / SO(p+1,q)_e$, except for item (iii),
due to the existence of non-cuspidal non-spherical discrete series, corresponding to the parameters $\lambda > 0$,
with $\mu_\lambda \in 2 \Z+1$ and $\mu_\lambda < 0$.
\end{remark}

\begin{remark}
The conditions in item (vi) essentially state that $\A f$ is a Schwartz function if $f$ is perpendicular to all non-cuspidal discrete series.
The factor $\Delta_\rho$, however, cannot be avoided, except in the cases $d=1$ and $q-p$ odd.
\end{remark} 

\begin{remark}
For the exceptional case, only (ii), (iii) and (vi) are relevant. The spherical discrete series corresponds to
$\lambda=3 \,(\mu_\lambda=0)$ and $\lambda=1\, (\mu_\lambda=-2)$.
\end{remark}

\section{Proof of the main theorem for $p\ge q$}

\begin{proposition}\label{thm1-1}
Let $p\ge q$.
\begin{enumerate}
\item[(i)] Let $f\in C^\infty _R (G/H)$, for $R>0$. Then $\A f \in C_R^\infty (\R)$.

\item[(ii)] Let $f\in\Sch(G/H)$. Then $\A f \in \S (\R)$.

\end{enumerate}
\end{proposition}
\begin{proof}
Let $f\in C^\infty _R (G/H)$, for $R>0$. By (\ref{aexpN}) and (\ref{nforpgeq}), we have
\[
\cosh ^{2} (|a_s \exp (N _{(v,u^1),v,w})|) =
(\cosh s + 1/2 e^s |u'| ^2)^2 +|v|^2+ |e^sw|^2 \ge \cosh ^2 s,
\]
and thus $R f(s) =0$,
for $|s|> R$, which shows (i).

For (ii), let $f\in \Sch(G/H)$.
As before we have, for $n\in \N$,
\begin{equation}\label{int}
\int _{N^*} |f (a_s n^*H)| dn^*\le C \int _{N^*} a(a_s n^*)^{-\rho_{\mathfrak q}} (1+|a_s n^*|)^{-n} dn^*,
\end{equation}
where $C$ is a positive constant.

The integral in (\ref{int}) is bounded by
\begin{align*}
&\int _{\R ^{dp-qd} \times \R^{dq}\times \R^{d-1}} ((\cosh s + 1/2 e^s |u'| ^2)^2
+|v|^2+ |e^sw|^2)^{-\frac {dp+dq+2(d-1)}4}\\
&\qquad\qquad \times(1+\log (( \cosh s + 1/2 e^s |u'| ^2)^2 +|v|^2+ |e^sw|^2)^{1/2}))^{-n} du' dvdw
\end{align*}
\begin{align*}
&\le (\cosh s)^{-\frac {dp+dq+2(d-1)}2}
\int _{\R ^{dp-dq} \times \R^{dq}\times \R^{d-1}} (1 + (1/{\sqrt 2}(\cosh s)^{-\frac 12} e^{s/2} |u'|) ^2)^2\\
& \qquad \qquad +|(\cosh s)^{-1}v|^2+ |(\cosh s)^{-1}e^s w|^2)^{-\frac {dp+dq+2(d-1)}4}\\
&\qquad \times(1+\log ((1 + (1/{\sqrt 2}(\cosh s)^{-\frac 12} e^{s/2} |u'|) ^2)^2\\
& \qquad \qquad+|(\cosh s)^{-1}v|^2+ |(\cosh s)^{-1}e^s w|^2 )^{-n} du' dv dw,
\end{align*}
since $\log \cosh s\ge 0$.

Consider the substitutions ${\overline u} = 1/{\sqrt 2}(\cosh s)^{-\frac 12} e^{s/2} u'$,
${\overline v} =(\cosh s)^{-1}v$ and ${\overline w} =(\cosh s)^{-1}e^s w$. Then
$du' = ({\sqrt 2}(\cosh s)^{\frac 12} e^{-s/2})^{dp-dq} d{\overline u}$,
$dv = (\cosh s)^{dq} d{\overline v}$, and $dw = ((\cosh s )e^{-s})^{d-1} d{\overline w}$, and the above integral becomes
\begin{align*}
&={\sqrt 2}^{dp-dq}e^{-\frac {dp-dq+2(d-1)}2 s}
\int _{\R ^{dp-dq} \times \R^{dq}\times \R^{d-1}} ((1 + |{\overline u}|^2)^2 +|{\overline v}|^2 +
 |{\overline w}|^2)^{-\frac {dp+dq+2(d-1)}4}\\
 &\qquad \qquad \times
(1+\log ((1 + |{\overline u}|^2)^2 +|\overline v|^2 +|\overline w|^2))^{-n} d{\overline u} d{\overline v}
d{\overline w}\\
& \le C_{p,q} a_s^{-\rho _1},
\end{align*}
where $C_{p,q}$ is a constant only depending on $p$ and $q$.
The proposition follows using the $U(\mathfrak g)$-equivariance of the Radon transform from Theorem~\ref{thm1}\,(iii).
\end{proof}

Let $\Sch(G/H)_d = L^2(G/H)_d\cap \Sch(G/H)$ denote the
span of the discrete series in $\Sch(G/H)$.
\begin{proposition}
Let $p\ge q$.
Then $R f = 0$, for $f\in \Sch(G/H)_d$.
\end{proposition}
\begin{proof}
Let $f\in\Sch(G/H)_d$. Then $\A f$ belongs to $\S(A_{\mathfrak q})$ by Theorem~{\ref{thm1-1}}, but at the same time $\A f$ is also an eigenfunction of
$\Delta _{A_{\mathfrak q}}$ on $A_{\mathfrak q}$. We conclude that $\A f=0$, and thus $R f = 0$.
\end{proof}

\section{Proof of (i) - (v) of the main theorem for $q>p$}

Let $\psi _\lambda$ be a generating function for the discrete series with parameter $\lambda$.
We notice that $\mu _\lambda =\lambda-\frac 12(dq-dp)+1=\lambda +\rho_{\mathfrak q}-dq-(d-1)+1$.
\begin{proposition}\label{newproof}
Let $p < q$. Then $R \psi _\lambda =0$,  for $\mu _\lambda > 0$. For $\mu _\lambda \le 0$, we have
\begin{equation}\label{equality}
R \psi _\lambda (a_s) = \int _{N^*} \psi _\lambda (a_s n^*) dn^*= C e^{(\mu_\lambda-d)s},\qquad (s\in \R),
\end{equation}
where $C\ne 0$ is a constant depending on $p,q$.
\end{proposition}
\begin{proof}
We define a $K$-invariant function $\tilde {\psi} _\lambda$ as
$\tilde {\psi} _\lambda (k a_s H) =  (\cosh s)^{-\lambda -\rho_{\mathfrak q}}$.
Then by (\ref{aexpN}) and (\ref{nforplq}), the Radon transform $R \tilde {\psi} _\lambda (a_s)$ is
\begin{equation}\label{subst}
\int _{\R ^{dq-dp} \times \R^{dp}\times \R^{d-1}}( (\cosh s-1/2e^s |v'|^2)^2+ |v'|^2
+ |u|^2+|e^s w|^2) ^{-\frac {\lambda +\rho_{\mathfrak q}}2}dv'dudw.
\end{equation}
Substituting $\tilde v' = (e^s/(2 \cosh s))^{1/2}v'$, $\tilde u = (1/\cosh s) u$ and $\tilde w = (e^s/\cosh s) w$, this becomes
\begin{align*}
&=\left (\frac {e^s}{2 \cosh s}\right )^{-\frac {dq-dp}2}\left (\frac 1{\cosh s}\right )^{-dp}\left (\frac {e^s}{\cosh s}\right ) ^{-(d-1)}
(\cosh s)^{-\lambda - \rho_{\mathfrak q}}\\
&\,\,\times\int _{\R ^{dq-dp} \times \R^{dp}\times \R^{d-1}}\left ( (1- |\tilde v'|^2)^2
+\left (\frac 2{e^s \cosh s}\right ) |\tilde v'|^2 + |\tilde u|^2 + |\tilde w|^2\right )
^{-\frac {\lambda +\rho_{\mathfrak q}}2}d\tilde v'd\tilde u d\tilde w\\
&=2^{\frac {dq-dp}2} e^{-\frac {dq-dp+2(d-1)}2 s}
(\cosh s)^{-\lambda}\\
&\,\,\times\int _{\R ^{dq-dp} \times \R^{dp}\times \R^{d-1}}\left ( 1+ |\tilde v'|^4 - 2 (\tanh s)
|\tilde v'|^2 + |\tilde u|^2+|\tilde w|^2\right )
^{-\frac {\lambda +\rho_{\mathfrak q}}2}d\tilde v'd\tilde u d\tilde w.
\end{align*}

Using the substitution ${\tilde{\tilde{u}}} = (1+ |\tilde v^1|^4 - 2 (\tanh s)
|\tilde v'|^2)^{-1/2}\tilde u$, and likewise for $\tilde w$, the integral becomes
\begin{align*}
&
\int _{\R ^{dq-dp} \times \R^{dp}\times \R^{d-1}}\left ( 1+ |\tilde v^1|^4 - 2 (\tanh s)
|\tilde v'|^2 \right )
^{-\frac {\lambda +\rho_{\mathfrak q}}2 + \frac {dp}2+ \frac{d-1}2}\\
&\qquad\qquad \times (1+|{\tilde{\tilde{u}}}|^2+|{\tilde{\tilde{w}}}|^2)
^{-\frac {\lambda +\rho_{\mathfrak q}}2}
d\tilde v'd{\tilde{\tilde{u}}}d{\tilde{\tilde{w}}}\\
&\qquad =C\int _0 ^\infty
\left ( 1+ \xi ^4 - 2 (\tanh s)\xi ^2 \right )
^{-\frac {\lambda +\rho_{\mathfrak q}-dp-(d-1)}2}
\xi ^{dq-dp-1}d\xi \\
&\qquad=\frac 12 C\int _0 ^\infty
\left ( 1+ x ^2 - 2 (\tanh s) x  \right )^{-\frac {\lambda +\rho_{\mathfrak q}-dp-(d-1)}2}
x ^{\frac {dq-dp}2 -1}dx
\end{align*}
using polar coordinates, where $C$ is the positive constant given by
\begin{equation*}
C=\int _0 ^\infty\int _0 ^\infty(1+\eta ^2+\sigma ^2)^{-\frac {\lambda +\rho_{\mathfrak q}}2}
  \eta ^{dp-1} d\eta\sigma ^{d-2}d\sigma
< \infty.
\end{equation*}

From \cite[3.252(10)]{GR}, we get
\begin{align*}
\int _0 ^\infty&
\left ( 1+ x ^2 + 2 (\cos t)
x  \right )^{-\nu } x ^{\mu -1}dx =\\
&2^{\nu -\frac 12} (\sin t)^{\frac 12 -\nu} \Gamma (\nu +\frac 12)B(\mu,2\nu -\mu)
P^{\frac 12 - \nu} _{\mu - \nu -\frac 12} (\cos t).
\end{align*}
We also have
\begin{align*}
&P^{\frac 12 - \nu} _{\mu - \nu -\frac 12} (y)=\\
1/\Gamma (\nu +\frac 12)& \left (\frac {1+y}{1-y}\right )^{\frac 12(\frac 12 - \nu)}
{}_2 F_1 \left (-\mu +\nu +\frac 12,\mu -\nu +\frac 12; \nu +\frac 12; \frac 12 - \frac 12 y\right ).
\end{align*}
With $y = \cos t = -\tanh s$, for $0<t<\pi$, we get $\sin t = 1/\cosh s$, $1-y = e^s/\cosh s$ and
$1+y =e^{-s}/\cosh s$. Putting this together, we get
\begin{align*}
\int _0 ^\infty
\left ( 1+ x ^2 - 2 (\tanh s)
x \right )^{-\nu }& x ^{\mu -1}dx =B(\mu,2\nu -\mu)(2 (\cosh s) e^s)^{\nu - \frac 12}\\
&\times {}_2 F_1 \left (-\mu +\nu +\frac 12,\mu -\nu +\frac 12; \nu +\frac 12; \frac 1{1+e^{-2s}}\right ).
\end{align*}

With $\mu = \frac {dq-dp}2$ and $\nu = \frac {\lambda +\rho_{\mathfrak q}-dp-(d-1)}2$, we get
\begin{equation*}
R \tilde {\psi} _\lambda (a_s)=C_\lambda e^{-ds} (1+ e^{-2s}) ^{-\frac {\mu_\lambda}2}
 {}_2 F_1\left (\frac {\mu_\lambda}2,
1 - \frac {\mu_\lambda}2;\frac {\mu_\lambda+dq-dp}2 ;\frac 1{1+e^{-2s}}\right ),
\end{equation*}
where $C_\lambda$ is a positive constant depending on $p,q$ and $\lambda$.

The hypergeometric function
$z\mapsto {}_2 F_1 ( {\mu_\lambda}/2, 1 - {\mu_\lambda}/2; (\mu_\lambda+dq-dp)/2 ;z )$
is a polynomial of degree $-{\mu_\lambda}/2$ for
$\mu_\lambda \le 0$, and degree ${\mu_\lambda}/2-1$ for
$\mu_\lambda > 0$.
We thus immediately get (\ref{equality}) for $\mu_\lambda=0$.

Now let $\mu_\lambda=-2m<0$. We can
write $\psi _\lambda (a_s)$ as the sum
\begin{equation*}
\psi _\lambda (a_s)=(\cosh s)^{-\lambda-\rho_{\mathfrak q}}+\sum _{j=1}^{m} C_j (\cosh s)^{-(\lambda+2j) -\rho_{\mathfrak q}}; 
\end{equation*}
or
\begin{equation*}
\psi_\lambda = \tilde {\psi}_\lambda  + \sum_{j=1}^{m} C_j \tilde {\psi}_{\lambda+2j}.
\end{equation*}
It follows that $R\psi_\lambda(a_s)$ can be written as a sum
\begin{equation*}
R\psi_\lambda(a_s)= C_0 e^{-ds} (1+ e^{-2s}) ^{m}+ \sum _{j=0}^{m-1} C_je^{-ds}(1+ e^{-2s}) ^{j},
\end{equation*}
where $C_0$ is a non-zero constant corresponding to the factor $\tilde {\psi} _\lambda(a_s)=(\cosh s)^{-\lambda -\rho_{\mathfrak q}}$.
Thus, since we know that $R\psi_\lambda(a_s)$ is a linear combination of
$e^{(\lambda-\rho _1)s}=e^{(\mu_\lambda-d)s}$ and $e^{(-\lambda-\rho _1)s}$,
we get $R\psi_\lambda(a_s)=Ce^{(\mu_\lambda-d)s}$, for a non-zero constant $C$.

Let finally
$\mu_\lambda > 0$.
Then $|{\psi} _\lambda|\le \|\phi_{\mu _\lambda}\|_\infty\tilde {\psi} _\lambda$, and
$|R{\psi} _\lambda|\le \|\phi_{\mu _\lambda} \|_\infty R\tilde {\psi} _\lambda$. We have
\begin{equation}\label{inftyasym}
|R{\psi} _\lambda (a_s)|\le C_1 R\tilde {\psi} _\lambda (a_s)\le C_2 e^{-ds}
\qquad \text{for}\qquad s\to \infty,
\end{equation}
and
\begin{equation}\label{-inftyasym}
|R{\psi} _\lambda (a_s)|\le C_1 R\tilde {\psi} _\lambda (a_s)\le C_2 e^{(\mu_\lambda-d)s}
\qquad \text{for}\qquad s\to -\infty,
\end{equation}
for positive constants $C_i$. Since
$s\mapsto R\psi_\lambda(a_s)$ again is a linear combination of
$e^{(\mu_\lambda-d)s}$ and
$e^{(-\lambda-\rho _1)s}$, we see
from (\ref{inftyasym}) and (\ref{-inftyasym})
that $R\psi_\lambda =0$.
\end{proof}

Consider the cases where $p<q$ and $d(q-p) \le 2$, i.e.,\,the cases $(d,p,q)=(1,q-1,q), \, (d,p,q)=(1,q-2,q)$ and $(d,p,q)=(2,q-1,q)$.
In the first case $\mu_\lambda = \lambda + 1/2$, and $\mu_\lambda = \lambda $ in the last two cases.
This means that $\mu_\lambda > 0$ and (i) follows from Proposition~\ref{newproof}.

For the proof of (v), we need to consider the cases
where $p<q$ and $d(q-p) \le 1$, i.e.,\,the cases $(d,p,q)=(1,q-1,q)$. From \cite[Theorem 5.1~(iii)(a)]{AF-JS}
(recall that in that paper $p:=p-1$), we see that the Schwartz condition in the real case is also satisfied for $p=q-1$.

\section{Reduction to the real case ($d=1$)}

Some of our results above for the projective hyperbolic spaces could also be established from \cite[Theorem~5.2]{AF-JS} via the remark below.
However, we feel that the new and different presentation, and in particular the new proof of Proposition~\ref{newproof},
merits the space given.

Let $\Fi = \C, \, \Ha$, with $p\ge 0, \, q\ge 1$, and $d=\dim _\R \Fi$.
There is a natural projection
\begin{equation*}
\X (dp,dq,\R) \to \X(p,q,\mathbb F),
\end{equation*}
with a natural action of $\U( 1; \mathbb F)$ on $\X(dp, dq; \R)$. Let
\begin{equation*}
\mathcal{E}_\lambda (p,q,\mathbb F)= \{f \in C^\infty (\X (p,q,\mathbb F)) \mid \Delta f= (\lambda^2-\rho_{\mathfrak q}^2) f \},
\end{equation*}
then there is a $G$-homomorphism,
\begin{equation*}
\mathcal{E}_\lambda (p,q,\mathbb F)\to \mathcal{E}_\lambda (dp,dq,\mathbb R),
\end{equation*}
which is an isomorphism onto the $\U(1; \mathbb F)$ invariant functions in $\mathcal{E}_\lambda (dp,dq,\mathbb R)$. We refer to
\cite{HS} for more details.

Let $p'+1=d (p+1)$ and $q'+1 = d(q + 1)$. We note that
$\rho_{\mathfrak q} = \frac 12 (dp+dq+2(d-1))=\frac 12 (p'+q')=\rho_{\mathfrak q}' $.
Let $\tilde {\psi} _\lambda$ be as in the proof of Theorem~\ref{newproof}, and let
$R^d _{p,q}$ and $R^1 _{p',q'}$ denote the Radon transforms corresponding to the spaces
$\X(p+1,q+1,\mathbb F)$ and $\X (p'+1,q'+1,\R)$ respectively.
Using the substitution $w' = e^s w$ in (\ref{subst}), we get the identity:
\begin{equation}\label{redtoone}
R^d _{p,q} \tilde {\psi} _\lambda (a_s) = e^{-(d-1)s}R^1 _{p',q'} \tilde {\psi} _\lambda (a_s),\qquad (s\in \R),
\end{equation}
which shows that some of our results for the projective hyperbolic spaces follow from the real ($d=1$) case. 
Notice though, that the elements $a_s\in A_{\mathfrak q}$ on the left-hand and the right-hand side of (\ref{redtoone}) belong to 
different groups, and that the reduction only works for the Abelian part.
Similarly, we get
\begin{equation*}
\A^d _{p,q} \tilde {\psi} _\lambda = \A^1 _{p',q'}  \tilde {\psi} _\lambda.
\end{equation*}

\section{Proof of (vi) of the main theorem - A Closer study of $Rf$ near $+\infty$}

We want to prove that $\A(Df) \in \S (\R)$, for $f\in \Sch(G/H)$ and $d(q-p)>1$.
Although we believe this to be true in general, our proof, near $+\infty$, is only valid for the dense $G$-invariant subspace 
generated by the $K$-finite and $(K\cap H)$-invariant functions.

For $d=1$, \cite[Theorem 5.1~(iii)(b)]{AF-JS} yields that the Schwartz decay conditions are
satisfied near $-\infty$ for $\A(f)$, and thus also for $\A(Df)$. For $d>1$, the proof from \cite{AF-JS} is easily
adapted in the same way as formula (\ref{redtoone}), which leaves us to study $Rf$ near $+\infty$.
We will concentrate on the proof for $d=2,4$ below, 
with some comments on the $d=1$ case, and further remarks in Section~\ref{last section}.

Consider the subgroup $T$ given by
\[k_{\theta} =
\left(
\begin{array}
[c]{cccc}%
I_{p+1} & 0 & 0& 0 \\
0& \cos\theta & 0 & \sin\theta\\
0& 0 & I_{q-1} & 0 \\
0& -\sin\theta & 0 & \cos\theta
\end{array}
\right) .
\]
Then
\begin{equation*}
k_\theta a_t \cdot x_0 =
(\sinh t, 0,\dots,0;\sin\theta\cosh t,0,\dots,0,\cos\theta\cosh t).
\end{equation*}

We see that $H\supset K_1$, with $K_1$ normalizing $T$, and
$K_2=(K_2\cap H)T(K_2 \cap H)$, where $K_2 \cap H= U(q, \Fi) \times U(1, \Fi)$.
Furthermore $U(q, \Fi)$ centralizes $A$,
and as is easily seen, $(K\cap H) k_\theta w a_t H = (K\cap H)k_\theta a_t H$, for $w \in U(1, \Fi)$.
From this we deduce that
\begin{equation}\label{G-decomposition}
K\cap H = (K\cap H)^T(K\cap H)^{A_\mathfrak q}\quad\text{and}\quad G=(K\cap H)TAH,
\end{equation}
where $(K\cap H)^T$ and $(K\cap H)^{A_{\mathfrak q}}$ denote the centralizers of $T$ and $A_{\mathfrak q}$ in $K\cap H$ respectively.
It follows that a $K\cap H$-invariant function is uniquely determined by the values $f(k_\theta a_tH)$,
for $(\theta,t)\in [0,\pi]\times\R^+$.

From the equation $(K\cap H) k_\theta a_t H = (K\cap H)a_s n_{u,v',w} H$, we get
\begin{equation}\label{cos-ch1}
\begin{aligned}
(\ch t)^2 &= (\cosh s-1/2e^s |v'|^2)^2+ |v'|^2 + |u|^2+|e^s w|^2, \,\text{and}\\
(\cos \theta \ch t)^2 &= (\cosh s-1/2e^s |v'|^2)^2 + |e^s w|^2. 
\end{aligned}
\end{equation}
Let $x=|u| , y=|v'|$ and $z=e^s|w|$.
Let $v =-\sh s + 1/2\es y^2$, then $y^2 = 1 +2\e-s v - \de-s$, and
\begin{equation}\label{cos-ch2}
\begin{aligned}
(\ch t)^2 &= 1+x^2+v^2+z^2, \,\text{and}\\
(\cos \theta  \ch t)^2&= (v - \e-s)^2 + z^2.
\end{aligned}
\end{equation}
For $p=0$, the variable $x=0$ and the integration over $x$ disappears, 
and for $d=1$, the integration over $z$ disappears,
Furthermore, the equations (\ref{cos-ch1}) and (\ref{cos-ch2}) are slightly different in these cases, see
Section~\ref{last section}.

Consider a $K\cap H$-invariant function $f$ of irreducible $K$-type.
Then the function $k \mapsto f(ka_t)$ is a
zonal spherical function on $K/(K\cap H)$, a Jacobi polynomial in
$\cos \theta$, of even order.
We can thus decompose $f$ as a finite sum of functions of the form
$h(k_\theta a_t) = (\cos \theta)^m h(a_t)$, where $m$ is even and $h$ is an even function.

Define an auxiliary function $H$ by $H(\ch ^2 t)=(\ch t)^{-m}h(a_t )$.
Then, using the change of coordinates from before, we have:
$$
h(k_\theta a_t)= (\cos \theta \ch t)^m H(\ch ^2 t)= ((v-\e-s )^2 + z^2)^{\frac m2} H(1+x^2+v^2+z^2),
$$
where for each $N \in \N$,
$$
|H(1+x^2+v^2+z^2)| < C (1+x^2+v^2+z^2)^{-\frac{\rho_{\mathfrak q}+m}2} (1+\log(1+x^2+v^2+z^2))^{-N}.
$$

With the above substitutions, we find
\begin{equation*}
\begin{aligned}
Rh(s) = e^{-ds} \int_0^\infty \int_0^\infty \int_{-\sh s}^\infty & H(1 + x^2 + v^2+z^2)
((v-\e-s )^2 + z^2)^{\frac m2} \\
& \times (1 +2\e-s v - \de-s)^{\frac{\beta -1}2}  x^\alpha z^{d-2}\, dv \, dx \, dz,
\end{aligned}
\end{equation*}
where $\alpha=dp-1\, , \, \beta=d(q-p)-1>0$, i.e.,\,$\beta$ is a positive integer.

We have the following upper bound, for $s\ge 0$, since $\beta \ge 1$:
\begin{align*}
|Rh(s)|\le C e^{-ds} \int_0^\infty \int_0^\infty \int_{-\infty} ^\infty&
\frac{(1 + x^2 + v^2+z^2)^{-\frac{\rho_{\mathfrak q}+m}2}}
{(1+\log(1 + x^2 + v^2+z^2))^N}(1+v^2+z^2)^{\frac m2} \\
&\times (1+v^2)^{\frac{\beta -1}2} x^\alpha z^{d-2} \, dv\, dx\, dz < +\infty.
\end{align*}
Applying Lebesgue's theorem, we get
$$
\lim_{s\to\infty} e^{ds} Rh(s) =
\int_0^\infty \int_0^\infty \int_{-\infty}^\infty
H(1 + x^2 + v^2+z^2)(v^2+z^2)^{\frac m2}  x^\alpha z^{d-2}\, dv\, dx\, dz.
$$

For convenience, we replace $z$ by $u$. We can define $Rh(s)$ as a function of the variable $z=\e-s $ near $z=0$,
for $z>0$. Let $F(z) = e^{ds} Rh(s)$, then
\begin{equation*}
\begin{aligned}
F(z) = \int_0^\infty\int_0^\infty \int_{\frac 12(z-z^{-1})}^\infty& H(1 + x^2 + v^2+u^2)
((v-z)^2 +u^2)^{\frac m2}\\
&\times (1 +2 z v - z^2)^{\frac{\beta -1}2}  x^\alpha u^{d-2} \, dv\, dx\, du.
\end{aligned}
\end{equation*}

Let $k_0$ be the largest integer such that $k_0< (\beta -1)/2 +1$, and $0\le k < k_0$. The derivatives $d^k/dz^k$ of
the integrand are zero at $v=-\sh s= \frac 12(z-z^{-1})$, whence the integrand is at least
$k_0$ times differentiable near $z=0$, and we can compute the
derivatives $d^k/dz^k F(z)$. For $k_0>0$, we will use Taylor's formula
to express $F(z)$ as a polynomial of degree $k_0-1$, plus a remainder term
involving $d^{k_0}/dz^{k_0}F(\xi)$, for some $0< \xi(z)< z$.

\begin{lemma}\label{diff kernel2}
Fix $v,u \in \R$, $m\in 2\Z^+$ and $\delta\in \frac 12 \Z^+ $, and define
$$
S(z)= S_{v,u,m,\delta}(z)= ((v-z )^2+ u^2)^{\frac m2} (1 +2 z v - z^2)^\delta.
$$
For $0\le j< \delta +1$, $d^j/dz^j S(z)$ is a polynomial in $((v-z )^2+ u^2), (v-z)$ and $(1 +2 z v - z^2)$,
of degree at most $m+\delta$ in $v$,  $m$ in $u$ and $m+2 \delta - j$ in $z$.
For $z=0$, the degree is at most $m+j$ in $v$, and $m$ in $u$.
When $j$ is odd, $d^j/dz^j S$ is an odd function of $v$ at $z=0$.
\end {lemma}

\begin{proof}
Straightforward, using that $d/dz (1 +2 z v - z^2) = - d/dz ((v - z)^2+u^2)= 2(v-z)$.
\end{proof}

Note, that for $d(q-p)$ odd, that is, $d=1$ and $q-p$ odd, the term $(\beta -1)/2= ((q-p)-2)/2$ is a half-integer,
and the statements in Lemma~\ref{diff kernel2} have to be changed accordingly.

Using Taylors formula, we get
$$
F(z)= c_0 + c_1 z + c_2 z^2 + \cdots + c_{k_0-1} z^{k_0-1} + R_{k_0}(\xi) z^{k_0},
$$
where $0 < \xi < z $, and
$$
 c_j = \frac 1{j!}\int_0^\infty \int_0^\infty \int_{-\infty}^\infty H(1 + x^2 + v^2+ u^2)
\frac{d^j}{dz^j}S_{v,u,m,(\beta-1)/2}(0)\, x^\alpha u^{d-2} \, dv\, dx\, du ,
$$
for $ j\in \{0, \dots, k_0-1\}$.
By Lemma~\ref{diff kernel2}, $c_j=0$, for $j$ odd.
The remainder term $R_{k_0}(\xi)$ is given by:
\begin{equation*}
 \frac 1{k_0 !}  \int_0^\infty \int_0^\infty \int_{\frac 12(\xi-\xi^{-1})}^\infty H(1 + x^2 + v^2+ u^2)
\frac{d^{k_0}}{dz^{k_0}}S_{v,u,m,(\beta-1)/2}(\xi)\, x^\alpha u^{d-2} \, dv\, dx\, du.
\end{equation*}

Consider $\A h(s)=e^{\rho_1 s} Rh(s) = z^{-(\rho_1 -d)} F(z)$, which is equal to
\begin{equation*}
c_0 z^{-(\rho_1 -d)}  + c_2 z^{-(\rho_1 -d-2)} + ... +c_{k_0-2} z^{-2}+ c_{k_0-1} z^{-1} + R_{k_0}(\xi).
\end{equation*}
The exponents $\rho_1 -d-2j= d(q-p)/2-1 -2j$, for $j\in \{0, \dots, k_0-1\}$, correspond to the parameters
$\lambda_1, \dots, \lambda_r$ of the non-cuspidal discrete series. From (\ref{exchange}), and the definition of
the differential operator $D$ in Theorem~\ref{Main-thm}~(vi),
$\A (D h)$ thus only has a possible contribution from the remainder term, and, due to the term $d^2/d s^2$, no
constant term at $\infty$.

Note, that for $d(q-p)$ odd, the last two terms are: $c_{k_0-1} z^{-\frac 12}+ z^{\frac 12} R_{k_0}(\xi)$, 
where the last term is
rapidly decreasing.
For the other cases, the constant term $C_{R_{k_0}}\ =\lim_{s \to \infty} R_{k_0}(\e-s)$ could be non-zero, but 
we will prove that $R_{k_0}(\xi) - C_{R_{k_0}}$ is rapidly decreasing at $+\infty$,
where $\xi = \xi(s)$, with $0 < \xi < \e-s$.
We also consider the case $k_0=0$, with $\xi=\e-s$.

Let $G(v,u,z)=  1/{k_0 !}\, d^{k_0}/dz^{k_0}S_{v,u,m,(\beta-1)/2}(z)$.
Then $G(v,u,z)-G(v,u,0)= z P(v,u,z)$, where $P$ is a polynomial of degree
less than $m+(\beta-1)/2$ in $v$, and less than $m$ in $u$.
Let $|P|$ and $|G|$ denote the polynomials defined
from $P$ and $G$ by taking absolute values in all coefficients.

Let in the following $C$ denote (possibly different) positive constants. With $0 < \xi < \e-s$, we get the following estimates
\begin{align*}
|R_{k_0}(\xi) -& C_{R_{k_0}}|\le \\
\e-s &\int_0^\infty\int_0^\infty \int_{-\infty}^\infty |H(1+x^2+v^2+u^2)||P|(v,u,1)\,x^\alpha u^{d-2}\, dv\, dx\, du \\
+ &\int_0^\infty\int_0^\infty \int_{-\infty}^{-\sh s}\ |H(1+x^2+v^2+u^2)||G|(v,u,0)\,x^\alpha u^{d-2}\, dv\, dx\, du.
\end{align*}
The first integral is bounded by $C\e-s$, since the double integral is convergent.
The second integral is bounded near infinity by $Cs^{-N}$, for all $N$, which is seen as follows.
For $s$ large, the integrand is for every $N \in \N$ bounded by
$C(x^2+v^2+u^2)^{-(\rho +m)/2} |v|^{m+k_0}\log(x^2+v^2+u^2)^{-N}$.
Substituting $v=-v$, $x=x'v$ and $u=u'v$, we have the estimates
\begin{align*}
&\le C\int_0^\infty\int_0^\infty\int^\infty_{\sh s}
(1+x'^2+u'^2)^{-\frac{\rho_{\mathfrak q}+m}2} v^{-\rho_{\mathfrak q}-m+m+k_0+\alpha+1+d-1}\\
&\qquad\qquad\qquad \times(\log(v^2)+\log(1+x'^2+u'^2))^{-N} \,x'^\alpha u'^{d-2}\, dv \, dx'\, du'\\
&\le C
 \int^\infty_{\sh s} v^{-\rho_{\mathfrak q}+k_0+\alpha+d}( \log(v))^{-N}\, dv.
\end{align*}
Inserting the values
$\rho_{\mathfrak q} = d(p+q)/2+d-1, \, k_0= (d(q-p)-2)/2$, and $\alpha=dp-1$, we end up with
$$
C \int^\infty_{\sh s} v^{-1} (\log (v))^{-N}\, dv = C(N-1)^{-1}(\log(\sh s))^{-N+1} \le C s^{-N+1}.
$$
It follows that $R_{k_0}(\xi) - C_{R_{k_0}}$ is rapidly decreasing at $+\infty$, whence $\A(Dh)$
is rapidly decreasing at $+\infty$, since the constant term is not present, which
finishes the proof of Theorem~\ref{Main-thm}~(vi) for $K$-irreducible $(K\cap H)$-invariant functions.

Finally, consider the $G$-invariant subspace $\mathcal V$ of $\Sch(G/H)$ generated
by the $K$-irreducible $(K\cap H)$-invariant functions. The conclusion in (vi) is clearly satisfied for $f \in \mathcal V$.
We need to show that $\mathcal V$ is dense in $\Sch(G/H)$.
Let $0\ne f \in L^2(G/H)$ be perpendicular to $\mathcal V$.
Let $\mathcal U$ be the closed $G$-invariant subspace of $L^2(G/H)$ generated by $f$.
Then $\mathcal U$ contains a non-zero $C^\infty$-vector $f_1 \in \Sch(G/H)$, and after a translation,
we may assume that $f_1(eH) \ne 0$.
The function $f_2$ defined by $0\ne f_2(gH)= \int_{K\cap H} f_1(kgH) \,dk$ is then a
$(K\cap H)$-invariant element in $\mathcal U$, belonging to
the closure of $\mathcal V$, which is a contradiction.

\section{Final Remarks - the remaining cases}\label{last section}

Theorem \ref{Main-thm} also holds for the real non-projective space $G/H=SO(p+1,q+1)_e/SO(p+1,q)_e$, except for item (iii).
The statements (i), (ii), (iv) and (v) are proved in \cite{AF-JS}.
For the proof of (vi), the last equations in (\ref{cos-ch1}) and (\ref{cos-ch2}) should be replaced by
\begin{align*}
\cos \theta \ch t &= \cosh s-1/2e^s |v'|^2 ,\,\text{and}\\
\cos \theta  \ch t &= (v - \e-s).
\end{align*}
Then $(\theta, t) \in [0,2\pi]\times \R^+$, and $m$ could be odd.
For $p=0$, the first equations in (\ref{cos-ch1}) and (\ref{cos-ch2}) should be replaced by:
\begin{align*}
\sh t &= \sh s - 1/2 \es v'^2,\,\text{and}\\
\sh t &= -v ,
\end{align*}
with $(\theta, t) \in [0,2\pi]\times \R$, 
$H$ defined by $H(-\sh t)= (\ch t)^{-m} h(a_t)$, and
$$
|H(v)| < C(1+v^2)^{-\frac{\rho_{\mathfrak q}+m}2} (1+\log(1+v^2))^{-N}.
$$
With these remarks it is not difficult to modify Lemma \ref{diff kernel2}, and complete the proof.
Notice, that a priori all constants $c_j$ in the Taylor expansion could be non-zero.

Finally, we consider the exceptional case, with $\Fi =  \Ca$ (and $p=0, q=1, d=8)$.
We will show that the formulas
(\ref{G-decomposition}) and (\ref{cos-ch1})
are meaningful and true for this case as well.
The formula (\ref{G-decomposition}) was already shown to be true and used in \cite{FJO}.
We give a brief outline of the proof of (\ref{cos-ch1}).

According to \cite{MTK} and \cite{RT}, the exceptional group $G$ can be defined
by the automorphisms of a 27 dimensional Jordan Algebra $J_{1,2}$ parameterized by
$\xi, u \in \R^3 \times \Ca^3$, with basis $E_1, E_2, E_3, F_1, F_2, F_3$.
We denote an element in $J_{1,2}$ by $X(\xi, u)$.

The subgroups $H$ and $K$ are in fact equal to the stabilizers of $E_3$, respectively of $E_1$.
The subgroup $N$, which in this case equals $N^*$, is defined in \cite{MTK} as $u(y,z), \, y\in$Im($\Ca$) and $z \in \Ca$.
The subgroups $A = \{ a_t\}$ and $T = \{ k_\theta\}$ are also defined there.
In \cite{MTK} the expression $a_s u(y,z)E_3$ is calculated, and in \cite{RT}
the expression $k_\theta X(\xi, u)$ is calculated;
combining these two calculations $k_\theta a_t E_3 = k_\theta a_t u(0,0) E_3$ can be calculated.

Recall that $\xi_1, \xi_3$ are invariant under $K\cap H$. To derive the first formula in
(\ref{cos-ch1}), we only need to compare the first coordinates of $k_\theta a_t E_3$
and $k_\theta a_t u(y,z) E_3$, $\xi_1(s,y,z)= \xi_1(t, \theta)$; to derive the second
formula in (\ref{cos-ch1}), we only need to compare the third coordinates
of $k_\theta a_t E_3$ and $k_\theta a_t u(y,z) E_3$, $\xi_3(s,y,z)= \xi_3(t, \theta)$.
We have
\begin{align*}
 \xi_1(t, \theta)&= -(\ch {(2t)}-1)/2, \\
\xi_1(s,y,z)&=
 -\ch {(2s)} ((1-z)/2 + |z|^4/4 +|y|^2)\\
& \qquad+ \sh {(2s)}(1/2 |z|^2(1-|z|^2/2) - |y|^2) + (1-|z|^2)/2,  \\
 \xi_3(t, \theta)&= (\ch {(2t)}+ 1)/4 (1+ \cos {(2\theta)}),\\
 \xi_3(s,y,z)&= 
\ch {(2s)} ((1-z)/2 + |z|^4/4 +|y|^2)\\
&\qquad - \sh {(2s)}(1/2 |z|^2(1-|z|^2/2) - |y|^2) + (1-|z|^2)/2. 
\end{align*}
A tedious, but straightforward calculation, leads to the formulas (\ref{cos-ch1}), with
$v'$ replaced by $z$, and $w$ replaced by $y$.

\bibliographystyle{amsplain}

\begin{thebibliography}{A}

\bibitem{AF-JS}
N.B.~Andersen, M.~Flensted-Jensen and H.~Schlichtkrull,
{Cuspidal discrete series for semisimple symmetric spaces},
\textit{J.\ Funct.\ Anal.\ }{\bf 263} (2012), 2384--2408.

\bibitem{vdb} E.P.~van den Ban, {The principal series for a reductive symmetric space. II. Eisenstein integrals},
\textit{J.\ Funct.\ Anal.\ }{\bf 109} (1992), 331--441.

\bibitem{FJ} M.~Flensted-Jensen, {Discrete Series for Semisimple Symmetric Spaces},
\textit{Ann.\ of Math.\ (2)\ }{\bf 111} (1980), 253--311.

\bibitem{FJO}
M.~Flensted-Jensen and K.~Okamoto,
An explicit construction of the $K$-finite vectors in the
discrete series for an isotropic semisimple symmetric space,
\textit{M\'em. Soc. Math. France (N.S.) }{\bf 15} (1984), 157--199.

\bibitem{GR}
I.~S.~Gradshteyn, I.~M.~Ryzhik, {\it Table of integrals, series, and products},
Academic press, San Diego, USA, sixth edition, 2000.

\bibitem{MTK}
M.~T.~Kosters,
{Spherical Distributions on Rank One Symmetric Spaces}
Thesis, Leiden, 1983.


\bibitem{HS}
H.~Schlichtkrull,
Eigenspaces of the Laplacian on hyperbolic spaces:
composition series and integral transforms,
\textit{J. Funct. Anal. }{\bf 70} (1987), 194--219.

\bibitem{RT}
R.~Takahashi,
Quelque r\'esultats sur l'Analyse Harmonique dans l'epace sym\'etrique non compact
de rang 1 du type exceptionelle,
\textit{Lecture Notes in Math.} {\bf 739} (1979), 511--567.
\end{thebibliography}

\end{document}